\documentclass[12pt]{amsart}
\usepackage{amsmath,amssymb,amsbsy,amsfonts,amsthm,latexsym,
                        amsopn,amstext,amsxtra,euscript,amscd,mathrsfs,color,bm}
                        
\usepackage{float} 
\usepackage[english]{babel}
\usepackage{url}
\usepackage[colorlinks,linkcolor=blue,anchorcolor=blue,citecolor=blue]{hyperref}

\begin{document}

\newtheorem{theorem}{Theorem}
\newtheorem{lemma}[theorem]{Lemma}
\newtheorem{example}[theorem]{Example}
\newtheorem{algol}{Algorithm}
\newtheorem{corollary}[theorem]{Corollary}
\newtheorem{prop}[theorem]{Proposition}
\newtheorem{definition}[theorem]{Definition}
\newtheorem{question}[theorem]{Question}
\newtheorem{problem}[theorem]{Problem}
\newtheorem{remark}[theorem]{Remark}
\newtheorem{conjecture}[theorem]{Conjecture}

\newcommand{\comm}[1]{\marginpar{%
\begin{color}{blue}
\vskip-\baselineskip 
\raggedright\footnotesize
\itshape\hrule \smallskip #1\par\smallskip\hrule\end{color}}}

\newcommand{\commI}[1]{\marginpar{%
\begin{color}{magenta}
\vskip-\baselineskip 
\raggedright\footnotesize
\itshape\hrule \smallskip I: #1\par\smallskip\hrule\end{color}}}
\def\xxx{\vskip5pt\hrule\vskip5pt}


\def\cA{{\mathcal A}}
\def\cB{{\mathcal B}}
\def\cC{{\mathcal C}}
\def\cD{{\mathcal D}}
\def\cE{{\mathcal E}}
\def\cF{{\mathcal F}}
\def\cG{{\mathcal G}}
\def\cH{{\mathcal H}}
\def\cI{{\mathcal I}}
\def\cJ{{\mathcal J}}
\def\cK{{\mathcal K}}
\def\cL{{\mathcal L}}
\def\cM{{\mathcal M}}
\def\cN{{\mathcal N}}
\def\cO{{\mathcal O}}
\def\cP{{\mathcal P}}
\def\cQ{{\mathcal Q}}
\def\cR{{\mathcal R}}
\def\cS{{\mathcal S}}
\def\cT{{\mathcal T}}
\def\cU{{\mathcal U}}
\def\cV{{\mathcal V}}
\def\cW{{\mathcal W}}
\def\cX{{\mathcal X}}
\def\cY{{\mathcal Y}}
\def\cZ{{\mathcal Z}}

\def\C{\mathbb{C}}
\def\F{\mathbb{F}}
\def\K{\mathbb{K}}
\def\G{\mathbb{G}}
\def\Z{\mathbb{Z}}
\def\R{\mathbb{R}}
\def\Q{\mathbb{Q}}
\def\N{\mathbb{N}}
\def\M{\textsf{M}}
\def\U{\mathbb{U}}
\def\PP{\mathbb{P}}
\def\A{\mathbb{A}}
\def\p{\mathfrak{p}}
\def\n{\mathfrak{n}}
\def\X{\mathcal{X}}
\def\x{\textrm{\bf x}}
\def\w{\textrm{\bf w}}
\def\ovQ{\overline{\Q}}
\def\rank#1{\mathrm{rank}#1}
\def\wf{\widetilde{f}}
\def\wg{\widetilde{g}}

\def\({\left(}
\def\){\right)}
\def\[{\left[}
\def\]{\right]}
\def\<{\langle}
\def\>{\rangle}

\def\gen#1{{\left\langle#1\right\rangle}}
\def\genp#1{{\left\langle#1\right\rangle}_p}
\def\genPs{{\left\langle P_1, \ldots, P_s\right\rangle}}
\def\genPsp{{\left\langle P_1, \ldots, P_s\right\rangle}_p}

\def\e{e}

\def\eq{\e_q}
\def\fh{{\mathfrak h}}

\def\lcm{{\mathrm{lcm}}\,}

\def\({\left(}
\def\){\right)}
\def\fl#1{\left\lfloor#1\right\rfloor}
\def\rf#1{\left\lceil#1\right\rceil}
\def\mand{\qquad\mbox{and}\qquad}

\def\jt{\tilde\jmath}
\def\ellmax{\ell_{\rm max}}
\def\llog{\log\log}

\def\m{{\rm m}}
\def\ch{\hat{h}}
\def\GL{{\rm GL}}
\def\Orb{\mathrm{Orb}}
\def\Per{\mathrm{Per}}
\def\Preper{\mathrm{Preper}}
\def \S{\mathcal{S}}
\def\vec#1{\mathbf{#1}}
\def\ov#1{{\overline{#1}}}
\def\Gal{{\rm Gal}}

\newcommand{\bfalpha}{{\boldsymbol{\alpha}}}
\newcommand{\bfomega}{{\boldsymbol{\omega}}}

\newcommand{\Ch}{{\operatorname{Ch}}}
\newcommand{\Elim}{{\operatorname{Elim}}}
\newcommand{\proj}{{\operatorname{proj}}}
\newcommand{\h}{{\operatorname{h}}}

\newcommand{\hh}{\mathrm{h}}
\newcommand{\aff}{\mathrm{aff}}
\newcommand{\Spec}{{\operatorname{Spec}}}
\newcommand{\Res}{{\operatorname{Res}}}

\numberwithin{equation}{section}
\numberwithin{theorem}{section}

\title{On roots of unity in orbits of rational functions}

\author{Alina Ostafe}
\address{School of Mathematics and Statistics, University of New South Wales, Sydney NSW 2052, Australia}
\email{alina.ostafe@unsw.edu.au}

\subjclass[2010]{11R18, 37F10}

\begin{abstract} 
In this paper we characterise univariate rational functions over a number field $\K$ having infinitely many points in the cyclotomic closure $\K^c$ for which the orbit contains a root of unity. Our results are similar to previous results of Dvornicich and Zannier describing all polynomials having infinitely many preperiodic points in $\K^c$.
\end{abstract}

\maketitle

\section{Introduction and statements}

\subsection{Motivation}
For a rational function $h\in \K(X)$ over a number field $\K$, we define the $n$th iterate of $h$ by
$$
h^{(0)}=X,\quad h^{(n)}=h(h^{(n-1)}),\quad n\ge 1.
$$

For an element $\alpha\in \ov\Q$ we define the {\it orbit} of $h$ at $\alpha$ as the set
\begin{equation}
\label{eq:Orb h}
\Orb_h(\alpha)=\{\alpha_n\mid \alpha_0=\alpha \ \text{and} \ \alpha_n=h(\alpha_{n-1}), \ n=1,\ldots\}.
\end{equation}
We use the convention that the orbit terminates if for some $n\ge 0$, $\alpha_n$ is a pole of $h$
and, in this case, $\Orb_h(\alpha)$ is a finite set. 

If the point $\alpha_n$ in~\eqref{eq:Orb h} is defined, then
$\alpha_0$ is not a pole of $h^{(n)}$ and
$\alpha_n=h^{(n)}(\alpha_0)$. However, the fact that the evaluation
$h^{(n)}(\alpha_0)$ is defined does not imply the existence of
$\alpha_n$, since this latter point is defined if and only if all the
previous points of the orbit~\eqref{eq:Orb h} are defined and
$\alpha_{n-1}$ is not a pole of $h$.  For instance, let 
$h(X)=1/X$. Then $h^{(2)}(X)=X$ and we see that $h^{(2)}(0)=0$, but
$\alpha_2=h(h(0))$ is not defined as $0$ is a pole for~$h$.  

In this paper we are looking at the presence of roots of unity in orbits~\eqref{eq:Orb h} of univariate rational functions. In particular, we prove that unless the rational function is very special, there are finitely many 
initial points that are roots of unity such that the corresponding orbit contains another 
root of unity. In fact our result is more general. 

This work is motivated by a similar result of Dvornicich and Zannier~\cite[Theorem 2]{DZ}  that applies 
only for preperiodic points of univariate polynomials (for a multivariate characterisation see~\cite[Theorem 34]{KaSi}). Our methods follow the same ideas and technique of~\cite{DZ}, including a sharp Hilbert's irreducibility theorem for cyclotomic extensions~\cite[Corollary 1]{DZ}  and an extension of Loxton's result~\cite[Theorem 1]{Loxt} in representing cyclotomic integers as short combination of roots of unity~\cite[Theorem L]{DZ}. 
We note that a more general Hilbert irreducibility theorem over cyclotomic fields has been obtained by Zannier~\cite[Theorem 2.1]{Za}, which may be of use for further generalisations of the results of~\cite{DZ} and of this paper. In Section~\ref{sec:main} we suggest such a generalisation  that covers these results in a unified scenario.

\subsection{Notation, conventions and definitions} 

We use the following notations:
\begin{itemize}
\item[] $\ovQ$: the  algebraic  closure of $\Q$;
\item[ ] $\U$: the set of all roots of unity in $\C$; 
\item[] $\K$: number field;
\item[] $\K^c=\K(\U)$: the cyclotomic closure of $\K$; 
\item[] $C(\alpha)$: the maximum of absolute values of the conjugates over $\Q$ of an algebraic number $\alpha$; 
\item[] $\K^c_A$: the set of $\alpha\in \K^c$ such that $C(\alpha)\le A$;
\item[] $T_d$: the Chebyshev polynomial of degree $d$; it is uniquely determined by the equation $T_d(x+x^{-1})=x^d+x^{-d}$.
\end{itemize}

\begin{definition} [\sl{Special rational functions}]
\label{def:Srf}
We call a rational function $h\in\K(X)$  {\sl special} if $h$ is a conjugate (with respect to the group action given by $\mathrm{PGL}_2(\K)$ on $\K(X)$) to $\pm X^d$ or
to $\pm T_d(X)$. 
\end{definition}

For a rational function $h\in\K(X)$, we define
\begin{equation}
\label{eq:Sh}
\cS_h=\{\alpha\in \K^c \mid \ \alpha_k\in\U \textrm{ for some } k\ge 1\},
\end{equation}
where $\alpha_k$ is defined as in~\eqref{eq:Orb h}.

As remarked before, if $\alpha_k\in\U$ for some $k\ge 1$, then $\alpha_k=h^{(k)}(\alpha)$, and $h^{(\ell)}(\alpha)$ is well defined, that is, $h^{(\ell)}(\alpha)$ is not a pole for $h$, for all $\ell<k$.

\subsection{Main results}
\label{sec:main}

Our goal is to prove that the set $\cS_h$ is finite unless the rational function $h$ is special. 
Let  $\Per(h)$ be the set of the periodic points of $h$ (that is points which generate 
purely periodic orbits) and let  $\Preper(h)$  be the set of the pre-periodic points of $h$ (that is, points  that have a periodic point in their orbits).

If $f$ is a polynomial, then the finiteness (under some natural conditions) of the set $\Per(f)\cap \U$, which is a subset of $\cS_f$, 
follows immediately as a very special case from a more general result of  
Dvornicich and Zannier~\cite{DZ}. 
More precisely, 
let $f\in \K[X]$ be of degree at least $2$. Then, by~\cite[Theorem~2]{DZ}, the set $\Preper(f)\cap \K^c$ is finite unless, for some linear polynomial $L\in\ov \Q[X]$  and for some $\varepsilon=\pm 1$, $(L\circ f\circ L^{-1})(X)$ is either $(\varepsilon X)^d$ or $T_d(\varepsilon X)$.

Here we extend the finiteness property of  $\Per(f)\cap \U$ to the full set $S_f$, as well as obtain such a result for a large class of 
{\it non-special\/} rational functions, see Definition~\ref{def:Srf}.

\begin{theorem}
\label{thm:KcOrb}
Let $h=f/g\in\K(X)$, where $f,g\in\K[X]$ with $\deg f=d$ and $\deg g=e$. Assume that $d-e>1$
and that $\max\{d,e\}\ge 2$.
If $f(X)-Y^mg(X)$ as a polynomial in X does not have a root in $\K^c(Y)$ for all positive integers $m\le \deg_X f$, 
then $\cS_h$ is finite unless $h$ is special. \end{theorem}

It is also natural to ask whether~\cite[Theorem~2]{DZ} can be extended in full and
thus investigate the finiteness of the set
\begin{equation}
\label{eq:Th}
\cT_h(A)=\{\alpha\in \K^c \mid \ \alpha_k \in\K_A^c   \textrm{ for some } k\ge 1\},
\end{equation}
where $\alpha_k$ is defined in~\eqref{eq:Orb h}.

For any rational function $h=f/g$ satisfying the degree conditions in Theorem~\ref{thm:KcOrb} there is a constant $L_h$ such that if $C(\gamma)> L_h$,
then the sequence $C(h^{(n)}(\gamma))$, $n =1, 2, \ldots$, is strictly monotonically
increasing, see 
Section~\ref{sizegrow} below.
Hence, we have $\Preper(h)\cap \K^c\subseteq \cT_h(L_h)$ and also $\cS_h = \cT_h(1)$.
Unfortunately some underlying tools seem to be missing in this situation.

Maybe one can start with a possibly easier problem, that is, for $a\in\N$ and a rational function $h\in\K(X)$, prove the finiteness of the set $$S_{h,a}=\{\alpha\in \K(\U_a) \mid \alpha_k\in\U_a \textrm{ for some } k\ge 1\},$$ where $\U_a=\{t\in\ovQ\mid t^n=a \textrm{ for some } n\ge 1\}$.

\section{Preliminaries}

\subsection{Hilbert's Irreducibility Theorem over $\K^c$}

We need the following result due to Dvornicich and Zannier~\cite[Corollary~1]{DZ}. We present it however in a weaker form that is needed for our results, but the proof is given within the proof of~\cite[Corollary~1]{DZ}.

\begin{lemma}
\label{lem:cor1DZ}
Let $f\in \K^c[X,Y]$ such that $f(X,Y^m)$ as a polynomial in X does not have a root in $\K^c(Y)$ for all positive integers $m\le \deg_X f$. Then $f(X,\zeta)$ has a root in $\K^c$ for only finitely many roots of unity $\zeta$.
\end{lemma}
\begin{proof} As we have mentioned, this statement is a part of the proof of~\cite[Corollary~1]{DZ}:
the polynomial $g$ which appears in the proof satisfies the same condition as our polynomial $f$
(we also note that here we alternated the roles of the variables $X$ and $Y$).
\end{proof}

We have the following straightforward consequence.
\begin{corollary}
\label{cor:corcor1DZ}
Let $f\in \K^c[X,Y]$ and a linear polynomial $\cL=aY+b\in \K^c[X]$ such that $f(X,\cL(Y^m))$ as a polynomial in X does not have a root in $\K^c(Y)$ for all positive integers $m\le \deg_X f$. Then $f(X,\beta)$ has a root in $\K^c$ for only finitely many elements $\beta\in \cL(\U)$.
\end{corollary}
\begin{proof}
For any $\beta=a\zeta+b\in \cL(\U)$, with $\zeta\in\U$, $f(X,\beta)$ has a root in $\K^c$ if and only if $g(X,\zeta)$ has a root in $\K^c$, where $g=f(X,aY+b)\in \K^c[X,Y]$.  The result now follows directly from Lemma~\ref{lem:cor1DZ} applied with the polynomial $g$. 
\end{proof}

\subsection{Representations via linear combinations of roots of unity}

We have the following extension of a result of Loxton~\cite{Loxt}, 
which we present in the form given by~\cite[Theorem L]{DZ}.

\begin{lemma}
\label{lem:loxton}
There exists a number $B$ and a finite set $E\subset \K$ with $\#E\le [\K:\Q]$  such that any algebraic integer $\alpha\in \K^c$ 
can be written as $\alpha=\sum_{i=1}^bc_i\xi_i$, where $c_i\in E$, $\xi_i\in\U$ and $b\le \# E\cdot \cR(BC(\alpha))$, where $\cR:\R\to \R$ is any Loxton function.
\end{lemma}

\subsection{The size of elements in orbits}
\label{sizegrow}
In this section we prove some useful simple facts about the size of iterates. Although these results follow by simple computations and may be found in other sources, we decided to present the proofs of these statements.

\begin{lemma}
\label{lem:growiter}
Let $f=X^d+a_{d-1}X^{d-1}+\cdots+a_0\in \K[X]$ be of degree $d\ge 2$ 
and let  $\alpha\in \ov\Q$ be such that 
\begin{equation}
\label{eq:alphacondpoly}
|\alpha|_v>\max_{j=0,\ldots,d-1} \{1,|a_j|_v\}
\end{equation} 
for some non-archimidean absolute value $|\cdot|_v$ of $\K$ (normalised in some way and extended to $\ov\K=\ov\Q$). Then $\{|f^{(n)}(\alpha)|_v\}_{n\in\N}$ is strictly increasing.
\end{lemma}
\begin{proof}
The proof follows by induction on $n\ge 1$. For $n=1$ we need to prove that $|f(\alpha)|_v>|\alpha|_v$. 
We note that 
$$
 |\alpha^d-f(\alpha)|_v\le\max_{j = 0,\ldots, d-1} |a_j \alpha^{j}|_v = \max_{j = 0,\ldots, d-1} |a_j|_v |\alpha|_v^{j} <  |\alpha|_v^d,
$$
where the last inequality follows from~\eqref{eq:alphacondpoly}.
Hence, 
$$
|f(\alpha)|_v=\max\{ |\alpha^d-f(\alpha)|_v, |\alpha|_v^d\}=|\alpha|_v^d.
$$
The result now follows since $d\ge 2$ and thus $ |\alpha|_v^d >  |\alpha|_v$. 

We assume now the statement true for iterates up to $n-1$.
Hence for 
$\beta = f^{(n-1)}(\alpha)$ we have 
$$
|\beta|_v = |f^{(n-1)}(\alpha)|_v > \ldots > |f(\alpha)|_v > |\alpha|_v>\max_{j=0,\ldots,d-1} \{1,|a_j|_v\}.
$$
Using the same argument as for $n=1$ with $\beta$ instead of $\alpha$ we obtain 
$$
|f^{(n)}(\alpha)|_v  = |f(\beta)|_v > |\beta|_v = |f^{(n-1)}(\alpha)|_v, 
$$
which concludes the proof.
\end{proof}

From Lemma~\ref{lem:growiter} we also remark that any $\alpha\in\ovQ$ satisfying~\eqref{eq:alphacondpoly} is not a zero of the polynomial $f$. 

\begin{corollary}
\label{cor:ratfctgrow}
Let $h=f/g$, where $f,g\in\K[X]$ are defined by 
$$
f=X^d+a_{d-1}X^{d-1}+\cdots+a_0,\quad g=X^e+b_{e-1}X^{e-1}+\cdots+b_0.
$$
Let $\alpha\in \ov\Q$ be such that 
\begin{equation}
\label{eq:alphacondrat}
|\alpha|_v>\max_{\substack{0\le i\le d\\0\le j\le e}} \{1,|a_i|_v,|b_j|_v\}
\end{equation}
for some non-archimidean absolute value $|\cdot|_v$ of $\K$. 
If $d-e>1$, then $\{|h^{(n)}(\alpha)|_v\}_{n\in\N}$ is strictly increasing.
\end{corollary}

\begin{proof}
We proceed by induction over $n\ge 1$. For $n=1$ one has to prove that $|h(\alpha)|_v>|\alpha|_v$. From the proof of Lemma~\ref{lem:growiter} and the definition of $\alpha$, one has $|f(\alpha)|_v=|\alpha|_v^d$ and $|g(\alpha)|_v=|\alpha|_v^e$, and thus $|h(\alpha)|_v=|\alpha|_v^{d-e}$. In particular, from the remark after Lemma~\ref{lem:growiter} we know that $\alpha$ is not a poly of $h$.
Since $d-e>1$, the conclusion follows for $n=1$.

We assume now the statement true for iterates up to $n-1$.
Hence for 
$\beta = h^{(n-1)}(\alpha)$ we have 
$$
|\beta|_v = |h^{(n-1)}(\alpha)|_v > \ldots > |h(\alpha)|_v > |\alpha|_v>\max_{\substack{0\le i\le d\\0\le j\le e}}  \{1,|a_i|_v,|b_j|_v\}.
$$
Using the same argument as for $n=1$ with $\beta$ instead of $\alpha$ we obtain 
$$
|h^{(n)}(\alpha)|_v  = |h(\beta)|_v > |\beta|_v = |h^{(n-1)}(\alpha)|_v, 
$$
which concludes the proof.
\end{proof}

We remark that for the case $d-e\le 1$, Corollary~\ref{cor:ratfctgrow} does not hold. Indeed, let $\alpha\in\ov\Q$ satisfying~\eqref{eq:alphacondrat}. From the proof of  Corollary~\ref{cor:ratfctgrow} we have $|h(\alpha)|_v=|\alpha|_v^{d-e}$.

If $d=e$, then $|h(\alpha)|_v=1$, and thus $\{|h^{(n)}(\alpha)|_v\}_{n\in\N}$ is not strictly increasing (or decreasing). For example, assume $\max_{\substack{0\le i\le d\\0\le j\le e}}  \{|a_i|_v,|b_j|_v\}<1$, then $|h^{(n)}(\alpha)|_v=1$ for all $n\ge 1$.

If $d=e+1$, then one has $|h^{(n)}(\alpha)|_v=|\alpha|_v$, $n\ge 1$, and thus again it is not strictly increasing. 

If $d-e<0$, assume that $|a_0|_v,|b_0|_v>1$.  From the above we have $|h(\alpha)|_v=|\alpha|_v^{d-e}\le |\alpha|_v^{-1}$, and in particular, $|h(\alpha)|_v^{-1}\ge  |\alpha|_v$, which satisfies~\eqref{eq:alphacondrat}. Now, for the next iterate we have
$$
|h^{(2)}(\alpha)|_v=|h(\alpha)|_v^{d-e}\frac{|f^*(h(\alpha)^{-1})|_v}{|g^*(h(\alpha)^{-1})|_v},
$$ 
where $f^*(X)=X^df(X^{-1})$ and $g^*=X^eg(X^{-1})$  (thus, $a_0$ and $b_0$ become the leading coefficients of $f^*$ and $g^*$, repsectively). Since $|h(\alpha)|_v^{-1}$ satisfies~\eqref{eq:alphacondrat} and using the condition $|a_0|_v,|b_0|_v>1$, simple computations as in the proof of Lemma~\ref{lem:growiter} show that $|h^{(2)}(\alpha)|_v=|a_0|_v/|b_0|_v$, which does not tell us anything about being increasing or decreasing from the previous iterate $h(\alpha)$. 

Thus, although similar ideas might work for $d-e\le 1$, different conditions are needed to be in place to control the strict growth of iterates, which is essential for the proof of Theorem~\ref{thm:KcOrb}.

In the proof of Theorem~\ref{thm:KcOrb} we  
apply Lemma~\ref{lem:loxton} to iterates of a rational function applied in a point of $\K^c$. 
For this reason we need to control the growth of the house of such iterates, which is presented below.
However it is convenient to start with estimating the absolute value of the iterates.

\begin{lemma}
\label{lem:growabsval}
Let $h=f/g$, where $f,g\in\K[X]$ are defined by 
$$
f=X^d+a_{d-1}X^{d-1}+\cdots+a_0,\quad g=X^e+b_{e-1}X^{e-1}+\cdots+b_0.
$$
Let $\alpha\in \ov\Q$ be such that 
\begin{equation}
\label{eq:alphacond}
|\alpha|>1+\sum_{i=0}^{d-1}|a_i|+\sum_{j=0}^{e-1}|b_j|.
\end{equation}
If $d-e>1$, then $\{|h^{(n)}(\alpha)|\}_{n\in\N}$ is strictly increasing.
\end{lemma}
\begin{proof}
The proof goes again by induction over $n$. We prove only the case $n=1$, since the implication from $n-1$ to $n$ follows exactly the same lines. We have $|h(\alpha)|=\frac{|f(\alpha)|}{|g(\alpha)|}$.

We look first at $|f(\alpha)|$. As above, by the triangle inequality, we have
$|f(\alpha)|\ge |\alpha|^d - |f(\alpha)-\alpha^d|$. Since
$$
 |f(\alpha)-\alpha^d|=|a_{d-1}\alpha^{d-1}+\cdots+a_0|\le |\alpha|^{d-1}\sum_{i=0}^{d-1}|a_i|,
$$
where the last inequality follows since $|\alpha|\ge 1$, we conclude that
\begin{equation}
\label{eq:absvalpolyf}
|f(\alpha)|\ge |\alpha|^{d-1}\(|\alpha|-\sum_{i=0}^{d-1}|a_i|\).
\end{equation}

We also have that 
\begin{equation}
\label{eq:absvalpolyg}
|g(\alpha)|=|\alpha^e+b_{e-1}\alpha^{e-1}+\cdots+b_0|\le |\alpha|^e\(1+\sum_{j=0}^{e-1}|b_j|\).
\end{equation}

Putting together~\eqref{eq:absvalpolyf} and~\eqref{eq:absvalpolyg}, and recalling the initial assumption~\eqref{eq:alphacond}, we conclude that $|h(\alpha)|>|\alpha|$. The induction step from $n-1$ to $n$ follows the same way as for $n=1$
\end{proof}

\begin{corollary}
\label{cor:growhouse}
Let $h=f/g$, where $f,g\in\K[X]$ are defined by 
$$
f=X^d+a_{d-1}X^{d-1}+\cdots+a_0,\quad g=X^e+b_{e-1}X^{e-1}+\cdots+b_0.
$$
Let $A\in\R$ be positive and define
\begin{equation}
\label{eq:lh}
L_h=\max_{\sigma}\left\{1+\sum_{i=0}^{d-1}|\sigma(a_i)|+\sum_{j=0}^{e-1}|\sigma(b_j)|,A\right\},
\end{equation}
where the maximum
is taken over all embeddings $\sigma$ of $\K$ in $\C$.
Let $\alpha\in \ov\Q$ be such that 
 $C\(h^{(k)}(\alpha)\)\le A$ for some $k\ge 1$.
If $d-e>1$,  then $C\(h^{(\ell)}(\alpha)\)\le L_h$ for all $\ell<k$. 
\end{corollary}
\begin{proof}
It follows from Lemma~\ref{lem:growabsval}. Indeed, assume that $d-e>1$ and that $C\(h^{(\ell)}(\alpha)\)> L_h$ for some $\ell<k$. This means that there exists a conjugate of $h^{(\ell)}(\alpha)$, which we denote by
$\sigma\(h^{(\ell)}(\alpha)\)$, such that $|\sigma\(h^{(\ell)}(\alpha)\)|> L_h$. We note that $\sigma\(h^{(\ell)}(\alpha)\)=\sigma(h)^{(\ell)}(\sigma(\alpha))$, where $\sigma(h)$ is the rational function $h$ in which we replace the coefficients of $f$ and $g$ by $\sigma(a_i)$ and $\sigma(b_j)$, $i=0,\ldots,d-1$, $j=0,\ldots,e-1$. We apply now Lemma~\ref{lem:growabsval} with the rational function $\sigma(h)$ and the point $\sigma\(h^{(\ell)}(\alpha)\)$ satisfying $|\sigma\(h^{(\ell)}(\alpha)\)|> L_h$ to conclude that $\{\sigma\(h^{(\ell+n)}(\alpha)\)\}_{n\in\N}$ is strictly increasing. Thus, we obtain that $|\sigma\(h^{(k)}(\alpha)\)|>A$, which is a contradiction with $C\(h^{(k)}(\alpha)\)\le A$.
\end{proof}

\subsection{Growth of the number of terms in rational function iterates} The main result of the paper relies on the following result of Fuchs and Zannier~\cite[Corollary]{FZ} which says that the number of terms in the iterates $h^{(n)}$ of a rational function $h\in\K(X)$ goes to infinity with $n$ (see~\cite{Za07} for a previous result applying only to polynomials). Here is the more precise statement:

\begin{lemma}
\label{lem:FZ}
Let $q\in\K(X)$ be a non-constant rational function and let $h\in\K(X)$ be of degree $d\ge 2$. Assume that $h$ is not  special. Then, for any $n\ge 3$, $h^{(n)}$ is a ratio of polynomials having all together at least $\frac{1}{\log 5}\((n-2)\log d - \log 2016\)$ terms.
\end{lemma}

\section{Proof of Theorem~\ref{thm:KcOrb}}

We proceed first by bringing the rational function $h$ to a monic rational function (that is, both numerator and denominator are monic polynomials). Indeed, if $h=f/g=c\wf/\wg$ with $c\in\K^*$ and $\wf$ and $\wg$  monic polynomials, then there exists a linear polynomial $\cL=\mu X\in\ovQ[X]$ such that $h_{\mu}=\cL\circ h\circ \cL^{-1}$ is monic, that is $\mu$ is a solution to the equation $c\mu^{1-d+e}=1$.
Without loss of generality (enlarging the field $\K$ if necessary) 
we can assume that $\mu\in\K$.  

Since $h_{\mu}^{(k)}=\cL\circ h^{(k)}\circ \cL^{-1}$,  we can work with the monic rational function 
\begin{equation}
\label{eq:hmu}
h_{\mu}(X)=\mu \frac{f(\mu^{-1}X)}{g(\mu^{-1}X)}=\frac{f_{\mu}(X)}{g_{\mu}(X)},
\end{equation}
where
$$
f_{\mu}(X) = X^d+a_{d-1}X^{d-1}+\cdots+a_0,\quad g_{\mu}(X) = X^e+b_{e-1}X^{e-1}+\cdots+b_0.
$$ 
We are now left with proving the finiteness of the set
$$
\cS_{h,\mu}=\{\alpha\in \K^c \mid \ h_{\mu}^{(k)}(\alpha)\in \cL(\U) \textrm{ for some } k\ge 1\}
$$
(rather than of the set $\cS_h$).

The proof follows the approach of the proof of~\cite[Theorem~2]{DZ}, coupled with Corollary~\ref{cor:corcor1DZ}. Indeed, by~\eqref{eq:hmu} simple computations show that $f_{\mu}(X)-\cL(Y^m)g_{\mu}(X)$ as a polynomial in X does not have a root in $\K^c(Y)$ if and only if $f(X)-Y^mg(X)$ has the same property, which is satisfied by our hypothesis. Thus, we apply Corollary~\ref{cor:corcor1DZ} with the polynomial $f_{\mu}(X)-Yg_{\mu}(X)$ and get that there are finitely many $\beta\in \cL(\U)$ such that $f_{\mu}(X)-\beta g_{\mu}(X)$ has a zero in $\K^c$. This implies that there are finitely many $\beta\in \cL(\U)$ such that $h_{\mu}(X)-\beta$ has a zero in $\K^c$. Denote by $S$ the set of such $\beta\in \cL(\U)$.

It is sufficient to prove that, for any $\beta\in S$,  there are finitely many $\alpha\in \K^c$ such that $h_{\mu}^{(k)}(\alpha)=\beta$ for some $k\ge 1$.

Let $A_\mu=C(\mu)$ where $\mu\in\K$ is the coefficient of the above linear polynomial $\cL(X) =\mu X$. 
Thus, for any $\beta=\mu\xi\in \cL(\U)$ for some $\xi\in\U$, we have $C(\beta)=C(\mu\xi)=A_{\mu}$.

Let $L_h$ be a positive integer defined as in Corollary~\ref{cor:growhouse} with $A$ replaced by $A_{\mu}$.  Let also $M$
be  a sufficiently large positive integer, chosen to satisfy
\begin{equation}
\label{eq:M large}
M>\frac{B_{h,\K}\log 5+\log 2016}{\log \max\{d,e\}}+2,
\end{equation}
where $B_{h,\K}$ is defined below to be a constant depending only on $h$ and $\K$.

If $k\le M$, then obviously there are finitely many $\alpha\in \K^c$ such that $h_{\mu}^{(k)}(\alpha)=\beta$ for any $\beta\in S$.

We assume now $k>M$ and we denote 
$$
\cS_{h,\mu}(M)=\{\alpha\in \cS_{h,\mu}\mid \ h_{\mu}^{(k)}(\alpha)\in \cL(\U) \textrm{ for some } k>M\},
$$
where as above $\cL(X) = \mu X$.

By Corollary~\ref{cor:growhouse}, for any $\alpha\in \cS_{h,\mu}(M)$ we have that $C(h_{\mu}^{(r)}(\alpha))\le L_h$ for all $r=0,\ldots,M$.

Moreover, as in the proof of~\cite[Theorem~2]{DZ}, for any $\alpha\in \cS_{h,\mu}(M)$ and any non-archimidean place $|\cdot|_v$ of $\K$ (normalised in some way and extended to $\ov\K=\ov\Q$), we have that 
$$|h_{\mu}^{(r)}(\alpha)|_v\le \max\{1,|\mu|_v,|a_j|_v,|b_j|_v\}$$ 
for all $r=0,\ldots,M$, since otherwise, by Corollary~\ref{cor:ratfctgrow}, we have that $\{|h_{\mu}^{(r+n)}(\alpha)|_v\}_{n\in\N}$ is strictly 
increasing, which contradicts $h_{\mu}^{(k)}(\alpha)\in \cL(\U)$ (and thus $|h_{\mu}^{(k)}(\alpha)|_v=|\mu\xi|_v=|\mu|_v$ for some $\xi\in\U$) for some $k> M$. 
Hence, taking a positive integer $D_{h,\mu}$ such that $D_{h,\mu}a_i$ and $D_{h,\mu}b_j$, $i=0,\ldots,d-1$, $j=0,\ldots,e-1$, and $D_{h,\mu}\mu$ are all algebraic 
 integers,
 we conclude that  $$|D_{h,\mu}h_{\mu}^{(r)}(\alpha)|_v\le \max\{|D_{h,\mu}|_v,|D_{h,\mu}\mu|_v,|D_{h,\mu}a_j|_v,|D_{h,\mu}b_j|_v\}\le 1,$$ and thus $D_{h,\mu}h_{\mu}^{(r)}(\alpha)$ are all algebraic integers for any $\alpha\in \K^c$ and $r=0,\ldots,M$. 
 
 Applying now Lemma~\ref{lem:loxton} for $D_{h,\mu}h_{\mu}^{(r)}(\alpha)$, $r=0,\ldots,M$, there exist a positive integer 
$ B_{h,\K}$ 
and a finite set 
$E_{\K}$, depending only on $h_{\mu}$ and $\K$
such that, for every $\alpha\in \cS_{h,\mu}(M)$ and every integer $0\le r\le M$, we can write $h_{\mu}^{(r)}(\alpha)$ in the form 
\begin{equation}
\label{eq:iterlox}
h_{\mu}^{(r)}(\alpha)=c_{r,1}\xi_{r,1}+\cdots+c_{r,B_{h,\K}}\xi_{r,B_{h,\K}}, \quad r=0,\ldots,M,
\end{equation}
where $c_{r,i}\in E_{\K}$ and $\xi_{r,i}\in\U$.

Assume now that for $M$ satisfying~\eqref{eq:M large} the set $\cS_{h,\mu}(M)$ is infinite. Since for any $\alpha\in\cS_{h,\mu}(M)$, $h_{\mu}^{(r)}(\alpha)$ can be written in the form~\eqref{eq:iterlox}, and the set $E_{\K}$ is finite, we can pick an infinite subset $\cT_{h,\mu}(M)$ of $\cS_{h,\mu}(M)$ such that, for any $\alpha\in\cT_{h,\mu}(M)$, the $c_{r,i}\in E_{\K}$ in~\eqref{eq:iterlox} are fixed for $i=1,\ldots,B_{h,\K}$ and $r=0,\ldots,M$. In other words, the coefficients $c_{r,i}$ do not depend on $\alpha$.

As in the proof of~\cite[Theorem~2]{DZ}, we may use the first equation corresponding to $r=0$ to replace $\alpha$ on the left-hand side of~\eqref{eq:iterlox} and thus obtain
\begin{equation}
\label{eq:variterlox}
h_{\mu}^{(r)}\(\sum_{i=1}^{B_{h,\K}}c_{0,i}x_{0,i}\)=\sum_{i=1}^{B_{h,\K}}c_{r,i}x_{r,i}, \quad r=1,\ldots,M.
\end{equation}
We view the points $(\xi_{r,i})$, $i=1,\ldots,B_{h,\K}$, $r=0,\ldots,M$, as torsion points on the variety defined by the equations derived from~\eqref{eq:variterlox} in $\G_{m}^{B_{h,\K}(M+1)}$, and by our assumption, there are infinitely many such points.

Following the proof of ~\cite[Theorem~2]{DZ}, which is based on the {\it torsion points theorem\/} (see~\cite[Theorem~4.2.2]{BG}),  this leads to the following identities
\begin{equation}
\label{eq:ratiterlox}
h_{\mu}^{(r)}\(\sum_{i=1}^{B_{h,\K}}c_{0,i}\xi_{0,i}t^{e_{0,i}}\)=\sum_{i=1}^{B_{h,\K}}c_{r,i}\xi_{r,i}t^{e_{r,i}}, \quad r=1,\ldots,M,
\end{equation}
where $\xi_{r,i}$ are roots of unity and $e_{r,i}$ are integers, not all zero.

We denote by $q(t)=\sum_{i=1}^{B_{h,\K}}c_{0,i}\xi_{0,i}t^{e_{0,i}}$, and~\eqref{eq:ratiterlox} shows that 
the rational functions $h_{\mu}^{(r)}(q(t))$, $r =1, \ldots, M$, can be represented  by a rational function  with 
 at most $B_{h,\K}$ number of terms. 
We apply now Lemma~\ref{lem:FZ} and conclude that $$B_{h,\K}\ge \frac{1}{\log 5}\((M-2)\log \max\{d,e\} - \log 2016\),$$ which contradicts the choice of $M$ as in~\eqref{eq:M large}. This concludes the proof.

\section{Comments}
\label{sec:comm}

It would be of interest to extend the result of~\cite{DZ} and of this paper to the set~\eqref{eq:Th}. 
One important tool would be the Hilbert Irreducibility Theorem over $\K^c$, which for such extensions would mean to prove that under some natural conditions, a polynomial $g\in\K^c[X,Y]$ has the property that $g(X,\alpha)$ is reducible over $\K^c$ for finitely many $\alpha\in\K^c$ with $C(\alpha)\le A$. One way to start investigating such a result would be to use first the Loxton theorem and represent all such $\alpha$ in the form $\eta_1\xi_1+\ldots+\eta_B\xi_B$, where $\eta_i\in E$, $i=1,\ldots,B$, $B$  a positive number and $E$ a finite set that depend only on $\K$. 
We thus reduced the problem to proving that there exist finitely many tuples $(\xi_1,\ldots,\xi_B)\in\U^B$ such that $C(\eta_1\xi_1+\ldots+\eta_B\xi_B)\le A$ and $g(X,\eta_1\xi_1+\ldots+\eta_B\xi_B)$ is reducible over $\K^c$. Then, one can apply the multivariate version of Hilbert Irreducibility Theorem over cyclotomic fields with explicit specialisations at the set of torsion points of $\G_m^B$ due to Zannier, see~\cite[Theorem 2.1]{Za}.

\section*{Acknowledgements}
The author is very grateful to Umberto Zannier for suggesting that the technique used in the proof of~\cite[Theorem 2]{DZ} would work for proving the results obtained in this paper. The author would like to thank very much  Igor Shparlinski and Umberto Zannier for many useful and inspiring discussions, and for comments on an initial version of the paper. 

The research of A.~O. was supported by the 
UNSW Vice Chancellor's Fellowship.

\end{document}